\documentclass[bibother]{asl}
\usepackage[czech,english]{babel}
\usepackage{mathrsfs}
\usepackage{graphicx}
\usepackage{shuffle}
\usepackage[all]{xy}
\usepackage{braket}
\usepackage{enumerate}
\usepackage{thm-restate}
\usepackage{microtype}
\usepackage[noadjust]{cite}
\input{macros.mac}

\newtheorem{theorem}{Theorem}

\newtheorem{corollary}[theorem]{Corollary}
\newtheorem{lemma}[theorem]{Lemma}

\theoremstyle{definition}
\newtheorem{definition}[theorem]{Definition}
\newtheorem{notation}[theorem]{Notation}

\theoremstyle{remark}
\newtheorem{remark}[theorem]{Remark}
\newtheorem{fact}[theorem]{Fact}

\newtheorem{conjecture}{Conjecture}
\newtheorem*{claim*}{Claim}

\title{Modeling {\sf FO}-limits for monadically stable sequences}
\author{Samuel Braunfeld}
\email{braunfeld@cs.cas.cz}
\address[S.~Braunfeld]{The Czech Academy of Sciences, Institute of Computer Science, Pod Vod\'{a}renskou v\v{e}\v{z}\'{\i} 2, 182 00 Prague, Czech Republic; and \\ Computer Science Institute of Charles University (IUUK), Praha, Czech Republic}
\author{Jaroslav Ne{\v s}et{\v r}il}
\email{nesetril@iuuk.mff.cuni.cz}
\address[J.~Ne{\v s}et{\v r}il]{Computer Science Institute of Charles University (IUUK), Praha, Czech Republic}
\author{Patrice Ossona de Mendez}
\email{pom@ehess.fr}
\address[P.~Ossona de Mendez]{Centre d'Analyse et de Math\'ematiques Sociales (CNRS, UMR 8557),
	Paris, France and\\
	Computer Science Institute of Charles University (IUUK),
	Praha, Czech Republic}
\date{\today}
\thanks{
This paper is part of a project that has received funding from the European Research Council (ERC) under the European Union's Horizon 2020 research and innovation programme (grant agreement No 810115 -- {\sc Dynasnet}). Samuel Braunfeld is further supported by Project 24-12591M of the Czech Science Foundation (GA\v{C}R), and supported partly by the long-term strategic development financing of the Institute of Computer Science (RVO: 67985807).
\\
\includegraphics[width=.2\textwidth]{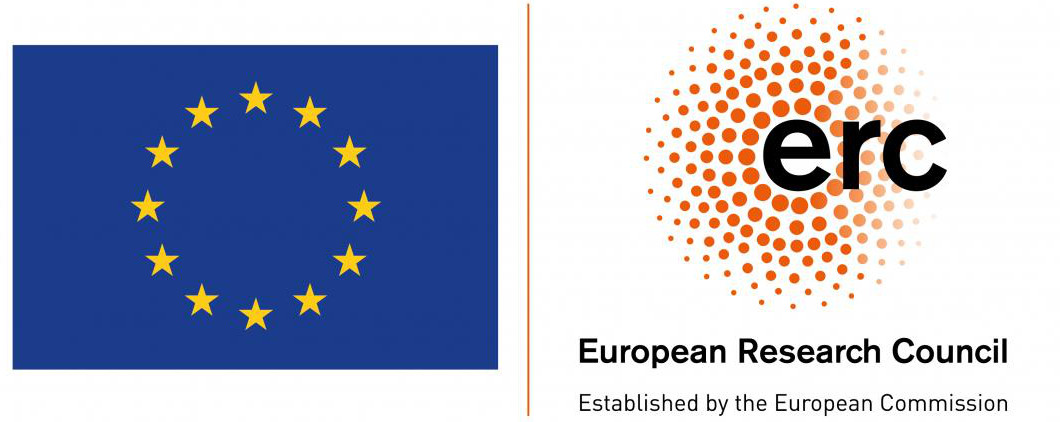}
}

\begin{document}
\begin{abstract}
  We show that given a monadically stable theory $T$, a sufficiently saturated $\mathbf M \models T$, and a coherent system of probability measures on the $\sigma$-algebras generated by parameter-definable sets of $\mathbf M$ in each dimension, we may produce a totally Borel structure $\tb \prec \mathbf M$ realizing these measures. Our main application is to prove that every {\sf FO}-convergent sequence of finite structures (with countable signature) from a monadically stable class admits a modeling limit. As another consequence, we prove a Borel removal lemma for monadically stable Lebesgue relational structures.
\end{abstract}
\maketitle
\section{Introduction}

The study of graph limits, and limits of more general structures, has attracted considerable interest over the last two decades, since the seminal papers  introducing the notions of left and local limits \cite{Aldous2006,borgchaylovas06:_count_graph_homom,Benjamini2001,Elek2007b}.
While the notion of \emph{left limit} is based on homomorphism counting and applies to general graphs (with non-trivial limits when the graphs are dense), the notion of \emph{local limit} was introduced for sequences of graphs with bounded degree. These two frameworks lead to seemingly essentially different analytic limit objects, \emph{graphons}, which are symmetric measurable functions from $[0,1]\times[0,1]$ to $[0,1]$ (up to measure preserving transformations) and \emph{graphings}, which are Borel graphs satisfying the so-called \emph{Mass Transport Principle}.

Several approaches have been considered to extend these frameworks~\cite{backhausz2022action,borgs2019L,kunszenti2019measures,kunszenti2022multigraph,kunszenti2024subgraph,CMUC,Nevsetvril2014}. 
From a model-theoretical point of view, a natural generalization of both left limits and  local limits is the notion of \emph{structural limit} \cite{CMUC}. In this setting, a fragment {\sf X} of first-order logic is fixed, and a sequence $(\mathbf M_n)_{n\in\mathbb N}$ of structures is said to be \emph{{\sf X}-convergent} if, for every formula $\varphi$ in $X$, the probability $\langle\varphi,\mathbf M_n\rangle$ that the formula $\varphi$ is satisfied in $\mathbf M_n$ converges (the free variables of $\varphi$ being interpreted as uniform and independent random  elements of the domain).  

{\sf FO}-convergence (where {\sf FO} denotes the full class of all first-order formulas) is a stronger notion of convergence than both left convergence (equivalent to {\sf QF}-convergence, where {\sf QF} denotes the fragment of quantifier-free formulas) and local convergence (equivalent to ${\sf FO}_1^{\rm local}$-convergence on bounded degree graph sequences, where ${\sf FO}_1^{\rm local}$ denotes the fragment of local formulas with a single free variable).

While the limit of an $\sf X$-convergent sequence can always be represented as a probability distribution over the Stone space dual to the Boolean algebra defined by ${\sf X}$ \cite{limit1}, it can also sometimes be represented by means of an analytic object, called a \emph{modeling}, which has stronger properties than mere Borel structures.
A \emph{totally Borel structure} is a structure whose domain is a standard Borel space, with the property that every definable set is Borel (in the appropriate product space).
A \emph{modeling} $\mathbf M$ is a totally Borel structure equipped with a Borel probability measure $\nu$ on its domain.
Note that, for every modeling $\mathbf M$ and every formula $\varphi$ (with $p$ free variables), the probability $\langle \varphi,\mathbf M\rangle$ of satisfaction of $\varphi$ in $\mathbf M$ when evaluated on a $\nu^{\otimes p}$-random $p$-tuple is given by
\[\langle \varphi,\mathbf M\rangle=\nu^{\otimes p}(\varphi(\mathbf M))=\idotsint 1_\varphi(v_1,\dots,v_p)\,{\rm d}\nu(v_1)\dots{\rm d}\nu(v_p),\]
where $\nu^{\otimes p}$ is the standard product measure and $1_\varphi(v_1,\dots,v_p)=1$ if $\mathbf{M}\models\varphi(v_1,\dots,v_p)$ and $1_\varphi(v_1,\dots,v_p)=0$, otherwise.

Based on these definitions, we say that a modeling $\mathbf L$
is a \emph{modeling {\sf X}-limit} of an {\sf X}-convergent sequence $(\mathbf M_i)_{i\in\mathbb N}$ of structures if, for  every $\varphi\in{\sf X}$, it holds that
\[
\langle\varphi,\mathbf L\rangle=\lim_{i\rightarrow\infty} \langle\varphi,\mathbf M_i\rangle.
\]

A natural problem is to determine when an {\sf X}-convergent sequence of graphs (or structures) admits a modeling {\sf X}-limit. In general, in order to have such a property, we either consider a sequence of graphs from a restricted class of graphs, or we consider a fragment more restricted than ${\sf FO}$ \cite{gajarsky2016first,grzesik2021strong,modeling,MapLim,modeling_jsl,limit1}.

When one considers {\sf FO}-convergence of sequences of graphs taken from a fixed monotone (i.e. closed on subgraphs) class of graphs $\mathscr C$, a solution to this problem was announced in \cite{modeling_jsl}, however the proof contains a gap, see ~\cite{corrigendum_jsl}. (In the next statement, the equivalence of the second and the third item follows from \cite{Adler2013}.)
\begin{theorem} \label{thm:monotone}
Let $\mathscr C$ be a monotone class of graphs (or, more generally, of binary structures). Then the following properties are equivalent:
\begin{enumerate}
    \item every {\sf FO}-convergent sequence of graphs in $\mathscr C$ has a modeling {\sf FO}-limit;
    \item the class $\mathscr C$ is nowhere dense;
    \item the class $\mathscr C$ is monadically stable.
\end{enumerate}
\end{theorem}

In this paper, we use a different approach to address the problem of the existence of a modeling limit, i.e. $(2) \Rightarrow (1)$ in the above theorem. In fact we prove a more general result. Note that  moving from the setting of monotone classes to hereditary classes, monadic stability becomes a more general notion than nowhere denseness. 

The main tool we develop here is a way to embed a modeling in any sufficiently saturated model of a monadically stable theory equipped with a coherent system of measures on parameter-definable sets in every dimension (see
Section~\ref{sec:mes struc}), as stated in the next theorem.

\begin{restatable}{theorem}{ThmTB}
	\label{thm:tb}
	Let $T$ be a monadically stable theory in a countable language, and let $\mathbf M$ be a measurable $\cont$-saturated model of $T$.
	
	Then, there exists a modeling $\tb$ and an elementary embedding $\zeta:\tb\hookrightarrow\mathbf M$, such that for every formula $\varphi(\bar x;\bar y)$ and every $\bar b\subset \tb$, the measure of $\varphi(\tb,\bar b)$ equals the measure
	of $\varphi(\mathbf M,\zeta(\bar b))$.
\end{restatable}

This theorem is not inherently about modeling limits, and applies also to measures that do not arise from the pseudofinite counting measure on a sequence of finite structures. However, our main application of this theorem is a proof of the existence of modeling {\sf FO}-limits in the more general setting of monadically stable classes. Particularly, this implies and generalizes Theorem \ref{thm:monotone}.

\begin{restatable}{theorem}{ThmMain}
\label{thm:main}
Let $\sigma$ be a countable signature, and 
let $\mathscr C$ be a monadically stable class of $\sigma$-structures. Then, every {\sf FO}-convergent sequence of $\sigma$-structures in $\mathscr C$ has a modeling {\sf FO}-limit.
\end{restatable}

As a further product of our construction, we obtain the following ``removal lemma''.
\begin{restatable}[Borel removal lemma for monadically stable structures]{theorem}{ThmRemoval}
\label{thm:removal}
Let $\mathbf L$ be a monadically stable Lebesgue relational structure (a structure on the standard probability space $[0,1]$, with all atomic relations Lebesgue measurable) in a countable language.
Then, there exists a Borel structure $\tb$, such that, for every quantifier-free formula $\varphi$, with $\varphi\rightarrow\bigwedge_{i<j<k}(x_i\neq x_j)$  we have
\begin{align*}
\langle\varphi,\tb\rangle&=\langle\varphi,\mathbf L\rangle, \text{ and}\\
\varphi(\tb)\neq\emptyset&\Rightarrow \langle\varphi,\tb\rangle>0.
\end{align*}
\end{restatable}

This theorem is motivated by the following conjecture.
\begin{conjecture}[L. Lov\' asz, B. Szegedy]
For every Lebesgue graph $\mathbf L$ (a graph on the standard probability space $[0,1]$ whose edge set is Lebesgue-measurable) there exists a Borel graph $\tb$, such that for every finite graph $F$, the density of $F$ in $\mathbf L$ and $\tb$ are equal, and every finite induced subgraph of $\tb$ has a positive density in $\tb$.
\end{conjecture}
The original conjecture was expressed in terms of random-free graphons (``Every random-free graphon $W$ is equivalent to a Borel graph $\mathbf G$ such that every finite induced subgraph of $\mathbf G$ has a positive density in $W$''). Note that the conjecture is known to hold if one only requires $\tb$ being Lebesgue.

\subsection{Proof sketch and structure of the paper} \label{sec:sketch}
We now roughly outline the proofs of Theorems \ref{thm:tb} and \ref{thm:main}.

In order to construct the modeling $\mathbf B$ in Theorem \ref{thm:tb}, we assume that the given measurable model $\mathbf M$ has a distinguished countable elementary submodel $\mathbf N$ (to be chosen later), and concern ourselves with getting the right measure on $\NN$-definable subsets of $\mathbf M$. The measure over $\mathbf M$ on such sets induces a measure on the space $S_1(\mathbf N)$ of 1-types over $\mathbf N$. Our first goal is to build the totally Borel $\mathbf B$ over $\mathbf N$ such that it contains a Borel set $A$ realizing every 1-type over $\mathbf N$, and so that all these realizations are forking-independent. We then aim to transport the measure on $S_1(\mathbf N)$  to $A$, where all of the measure of $\mathbf B$ will be concentrated; this is complicated by the fact that the topology on $A$ is finer than on $S_1(\NN)$, but can be done as a consequence of a general selection theorem. 

At this point, we have ensured that $\mathbf B$ gives the right measure to definable sets in dimension one, and we now want to argue that this ensures us to give the right measure to definable sets in all dimensions. This is where the choice of the countable $\mathbf N \prec \mathbf M$ becomes relevant. The structure theory for monadic stability shows that over any set, forking-dependence defines an equivalence relation on singletons. We show that we can choose $\mathbf N$ so that each class of this equivalence relation over $\mathbf N$ has measure 0 in $\mathbf M$. Thus, in $\mathbf M$, a random $k$-tuple will almost surely consist of points that are forking-independent over $\mathbf N$. The same is true in $\mathbf B$, since we have concentrated all our measure on the forking-independent set $A$. 

Stability theory (in particular, stationarity of types over models) then shows that if we have a $k$-tuple that is forking-independent over $\mathbf N$, the relations that hold of it are determined by the 1-types over $\mathbf N$ of the individual points. Since we have given the correct measure to 1-types, this shows we must also give the correct measure to $k$-types, and thus to definable sets in dimension $k$, finishing the proof of Theorem \ref{thm:tb}.

In order to prove Theorem \ref{thm:main}, we apply Theorem \ref{thm:tb} to the case where $\mathbf M$ is obtained as an ultraproduct of the given {\sf FO}-convergent sequence, equipped with the Loeb measure.

After general preliminaries in Section \ref{sec:prelim} and further preliminaries about type spaces in Section \ref{sec:Stone}, we show that we can build the totally Borel model nicely realizing all 1-types over any countable $\mathbf N \prec \mathbf M$ in Section \ref{sec:TB}. Then we show that we can choose $\mathbf N \prec \mathbf M$ as desired in Section \ref{sec:dismantle}. Finally, in Section \ref{sec:transfert}, we transport the measure from $\MM$ to $\tb$ and show that giving the right measures in dimension one lifts to all dimensions. This proves Theorem \ref{thm:tb}, and Theorems \ref{thm:main} and \ref{thm:removal} quickly follow. Further conjectures are given in Section \ref{sec:conc}.

\subsection{Acknowledgments}
We thank Chris Laskowski for helpful discussion about the construction of the totally Borel model in Section \ref{sec:TB}. We thank the referee for their thorough feedback, improving the correctness and readability of the paper.

\section{Preliminaries}
\label{sec:prelim}
\subsection{Measure theory}

Given a measure space $(X,\Sigma,\mu)$,
a \emph{$\mu$-null set} $A$ of $X$ is a $\mu$-measurable subset of $X$ with $\mu$-measure $0$.

Given two measurable spaces $(X_1,\Sigma_1)$
and $(X_2,\Sigma_2)$, a measurable mapping $f:X_1\rightarrow X_2$, and a measure 
$\mu:\Sigma_1\rightarrow [0,+\infty]$, the \emph{pushforward} of $\mu$ by $f$ is the measure $f_*(\mu)$ defined by
\[
f_*(\mu)(A)=\mu(f^{-1}(A)),
\]
for all $A\in\Sigma_2$.

A \emph{Polish space} $X$ is a  separable completely metrizable topological space.
A measurable space $(X,\Sigma)$
 is  a \emph{standard Borel space} if $X$ is a Polish space and 
$\Sigma$ is the Borel $\sigma$-algebra of $X$. According to Kuratowski's theorem, every standard Borel space is Borel isomorphic to one of $\mathbb R$, $\mathbb Z$, or a finite discrete space.

A \emph{Borel isomorphism} is a measurable bijective function between two standard Borel spaces whose inverse is also measurable. 

\begin{theorem}
    [Lusin-Suslin; see {\cite[Corollary 15.2]{Kechris1995}}]\label{fact:Suslin}
    Every Borel bijection between standard Borel spaces has a Borel inverse, hence is a Borel isomorphism.
\end{theorem}

 Measurable mappings between standard Borel spaces allow us to pull back a Borel measure, as stated in the next theorem.

\begin{theorem}[{\cite[Theorem 9.1.5]{Bogachev2006}}, stated for standard Borel spaces] \label{thm:Bog}
	Let $X$ and $Y$ be standard Borel spaces and let $f:X\rightarrow Y$ be a Borel mapping such that $f(X)=Y$. Then, for every Borel measure $\nu$ on $Y$, there exists a Borel measure $\mu$ on $X$ such that $\nu=f_*(\mu)$.
	\end{theorem}

Let $(X,\Sigma)$ be a measurable space, and let $\mu$ be a measure on this space. For a positive integer $k$, we denote by $\Sigma^k$ the \emph{product sigma algebra}, that is the  sigma algebra on $X^k$ generated by the subsets of the form $\prod_{i<k}B_i,$ where $B_i\in\Sigma$ for all $i<k$. The \emph{product measure} $\mu^{\otimes k}$ is defined as a measure on $(X^k,\Sigma^k)$ such that $\mu^{\otimes k}(\prod_{i<k} B_i)=\prod_{i<k}\mu(B_i)$ for all
$B_0,\dots,B_{k-1}\in\Sigma$. If $\mu$ is $\sigma$-finite (in particular, when $\mu$ is a probability measure), then the product measure $\mu^{\otimes k}$ is uniquely defined and, for every integrable $f$ we have (Fubini theorem):
\[
\int f(\bar x)\,{\rm d}\mu^{\otimes k}(\bar x)=\idotsint f(x_0,\dots,x_{k-1})\ {\rm d}\mu(x_0)\dots{\rm d}\mu(x_{k-1}).
\]
\subsection{Model theory} \label{sec:mod th}
Let $\sigma$ be a countable signature.
We define ${\sf FO}[\sigma]$ as the set of all first-order formulas in the language of $\sigma$-structures, where the free variables belong to the set $\set{x_i |  i<\omega}$. For $k\leq\omega$, we further define ${\sf FO}_k[\sigma]$ as the subset of ${\sf FO}[\sigma]$ with all the formulas whose free variables belong to $\set{x_i |  i<k}$. In particular, ${\sf FO}_0[\sigma]$ is the set of all the sentences in the language of $\sigma$-structures, and ${\sf FO}_\omega[\sigma]={\sf FO}[\sigma]$. When the signature $\sigma$ is implicit, we use the simplified notations ${\sf FO}$ and ${\sf FO}_k$ in place of ${\sf FO}[\sigma]$ and ${\sf FO}_k[\sigma]$.

For every $k\leq\omega$, the quotient of ${\sf FO}_k$ by the relation of logical equivalence defines a countable Boolean algebra. The Stone dual of this Boolean algebra is denoted by 
$S({\sf FO}_k)$. The points of  $S({\sf FO}_k)$ are the maximal consistent subsets of formulas in ${\sf FO}_k$, which are called \emph{complete theories} if $k=0$, and \emph{complete $k$-types} if $k>0$. 
The space $S({\sf FO}_k)$ is endowed with a topology, whose basis is its set of clopen sets (indexed by formulas in ${\sf FO}_k$)
\[
\clop(\varphi):=\set{\tau\in S({\sf FO}_k) |  \varphi\in\tau}.
\]
With this topology, $S({\sf FO}_k)$  is a Polish space (as the signature $\sigma$ is assumed to be countable). Hence, with the associated Borel sigma-algebra, $S({\sf FO}_k)$  is a \emph{standard Borel space}.

For a theory $T$, we define 
\begin{align*}
S_k(T)&=\set{p\in S({\sf FO}_k) |  T\subseteq p}.
\end{align*}
Note that $S_k(T)$ is a closed subset of $S({\sf FO}_k)$ since $S_k(T)=\bigcap_{\varphi\in T}\clop(\varphi)$.

Given a $\sigma$-structure $\mathbf M$, an ordinal $\alpha\leq\omega$, and a tuple $\bar a=(a_i :  i<\alpha)$ of elements of $\mathbf M$, i.e. $\bar a \in \mathbf{M}^\alpha$, we define the \emph{type} of $\bar a$ in $\mathbf M$ to be the
set of all formulas $\varphi\in {\sf FO}_\alpha$ such that 
$\mathbf M$ satisfies $\varphi$ when $x_i$ is interpreted as $a_i$ (for all  $i<\alpha$). Given $\MM \models T$ and $k \leq \omega$, we let $\tp^k \colon \MM^k \to S_k(T)$ be the function sending a $k$-tuple of elements of $\MM$ to its $k$-type. In order to simplify notation, we will often omit the $k$, so $\tp(\abar)$ will denote 
$\tp^{|\abar|}(\abar)$.

Also, given a $\sigma$-structure $\MM$, the \emph{theory of $\MM$}, denoted $\Th(\MM)$, is the set of $\sigma$-sentences satisfied by $\MM$.

We say that a substructure $\NN \subset \MM$ is an \emph{elementary substructure}, denoted $\NN \prec \MM$, if for every formula $\phi(\xbar)$ and $\nbar \subset N$, we have $\NN \models \phi(\nbar) \iff \MM \models \phi(\nbar)$. In particular, $\NN \models \Th(\MM)$. The \emph{Tarski-Vaught criterion} says that $\NN \prec \MM$ if, for every formula 
$\phi(x, \ybar)$ and every $\nbar \subset N$, $\MM \models \exists x\ \phi(x, \nbar)$ implies that  there is some $a \in \NN$ such that $\MM \models \phi(a, \nbar)$.

Given a structure $\MM$ and $A \subset M$, we let $\MM_A$ be the expanded structure in which all elements of $A$ are named by constants. Given a cardinal $\kappa$, $\MM$ is \emph{$\kappa$-saturated} if for every $A \subset M$ with $|A| < \kappa$, every type in $S(\Th(\MM_A))$ is realized by some tuple in $\MM_A$.

\subsection{Measurable structures} \label{sec:mes struc}

A \emph{measurable} $\sigma$-structure is a $\sigma$-structure $\mathbf M$ equipped, for each integer $k$, with a sigma algebra $\Sigma_k$ and a probability measure $\nu_k$ on $M^k$, such that 
\begin{itemize}
    \item every definable subset of $M^k$ is $\nu_k$-measurable;
    \item the probability measures $\nu_k$ satisfy a Fubini-like property: for every measurable bounded real-valued function $f$ on $M^k$, and every permutation $\rho$ on $\{0,\dots,k-1\}$ we have
\[
\int f(\bar x)\,{\rm d}\nu_k(\bar x)=\idotsint f(x_0,\dots,x_{k-1})\ {\rm d}\nu_1(x_{\rho(0)})\dots{\rm d}\nu_1(x_{\rho(k-1)}).
\]    
\end{itemize}

An important example of a measurable $\sigma$-structure 
is given by the ultraproduct $\prod_U \mathbf M_n$ where each $M_n$ is finite and equipped with the uniform probability measure, and the probability measures $\nu_k$ are obtained from the product probability measures of the finite structures using the ultraproduct construction of finite probability measure spaces \cite{loeb1975conversion}. We call these measures the \emph{Loeb measures} of the ultraproduct $\prod_U \mathbf M_n$.

In contrast with this complex structure,
a \emph{totally Borel model} is a $\sigma$-structure, whose domain is a standard Borel space  with the property that every definable set is Borel, and 
a \emph{modeling} $\mathbf M$ is a totally Borel $\sigma$-structure endowed with a Borel probability measure $\nu$, and every power of the structure is implicitly equipped with the product $\sigma$-algebra and product measure. 
Every modeling, equipped with the product Borel sigma algebras and the product probability measures $\nu^{\otimes k}$ is a measurable $\sigma$-structure. As shown in \cite{keisler2001loeb}, an ultraproduct equipped with the Loeb measures is far from being a totally Borel model.

\subsection{Structural convergence}
To a finite $\sigma$-structure $\mathbf M$ and a formula $\psi\in{\sf FO}$, we define $\langle\psi,\mathbf M\rangle$ as the probability that $\mathbf M$ satisfies $\psi$ for an interpretation of the free variables by random elements of $\mathbf M$ chosen uniformly and independently. (In the case where $\psi$ is a sentence, then  $\langle\psi,\mathbf M\rangle$ is $1$ if $\mathbf M$ satisfies $\psi$, and $0$ otherwise.)
For $p\leq\omega$, every finite $\sigma$-structure $\mathbf M$ injectively defines a unique measure $\mu_p^{\mathbf M}$ on $S({\sf FO}_p)$ such that $\mu_p^{\mathbf M}(\clop(\psi))=\langle\psi,\mathbf M\rangle$.

A sequence $(\mathbf M_n)_{n\in\mathbb N}$ of finite $\sigma$-structures is \emph{{\sf FO}-convergent} if, for every $\psi\in{\sf FO}$, the sequence $(\langle\psi,\mathbf M_n\rangle)_{n\in\mathbb N}$ has a limit. 
If $(\mathbf M_n)_{n\in\mathbb N}$ is {\sf FO}-convergent, then for each $p\leq\omega$, the probability measures $\mu_p^{\mathbf M_n}$ converge weakly to a probability measure $\mu_p$ such that, for every $\psi\in{\sf FO}_p$, we have
\[
\mu_p(\clop(\psi))=\lim_{n\rightarrow\infty}\langle\psi,\mathbf M_n\rangle.
\]
This notion was introduced in \cite{CMUC}. 

A \emph{modeling ${\sf FO}$-limit} of an {\sf FO}-convergent sequence $(\mathbf M_n)_{n\in\mathbb N}$ of finite $\sigma$-structures is a modeling $\mathbf L$, such that 
\[
\langle\psi,\mathbf L\rangle=\lim_{n\rightarrow\infty}\langle\psi,\mathbf M_n\rangle.
\]

The existence of modeling ${\sf FO}$-limits has been proved for sequences for several restricted classes of graphs (bounded height trees \cite{limit1}, trees \cite{modeling}, mappings \cite{MapLim},~etc.), and eventually for nowhere dense classes \cite{modeling_jsl}. However, the proof of 
\cite{modeling_jsl} is flawed in the crucial implication $(2) \Rightarrow (1)$, see \cite{corrigendum_jsl}. In this paper, we use different arguments to give a  proof of a generalization of this last result to all monadically stable classes of finite relational structures with countable signature.

\subsection{Stability, monadic stability, and forking} \label{sec:fork}

Here we recall the notions (monadic) stability and of forking independence, our main technical tool. Forking independence gives a notion of independence between two sets over a third, generalizing, for example, linear independence in vector spaces. For another example, in graphs of bounded degree, sets $A$ and $B$ are forking-independent over $C$ if every connected component containing a vertex from both $A$ and $B$ also contains a vertex from $C$. 

\begin{definition} \label{def:stable}
    A (possibly incomplete) theory $T$ is \emph{unstable} if there is a formula $\phi(\xbar; \ybar)$ and some $\MM \models T$ containing tuples $(\abar_i : i \in \omega), (\bbar_j : j \in \omega)$ such that $\MM \models \phi(\abar_i; \bbar_j) \iff i \leq j$. Otherwise, $T$ is \emph{stable}.

    A theory $T$ is \emph{monadically stable} if every theory obtained by expanding models of $T$ by any number of unary predicates remains stable.

    Given a class $\CC$ of structures in a fixed language, we say that \emph{$\CC$ is (monadically) stable} if $T_\CC = \bigcap_{\MM \in \CC} \Th(\MM)$ is. 
\end{definition}

In combinatorial contexts, it is common to define monadically stable classes as those that do not transduce the class $\mathcal{LO}$ of finite linear orders. We will not use this, but define it nevertheless.

\begin{definition}
    Let $\CC$ be a class of structures in a fixed language $\sigma$. Let $\sigma^+$ be the expansion of $\sigma$ by countably many unary predicates, and let $\CC^+$ be the class of all $\sigma^+$ structures whose $\sigma$-reduct is in $\CC$. We say \emph{$\CC$ transduces $\mathcal{LO}$} if there is some $\sigma^+$-formula $\phi(x;y)$ (on singletons) and structures $(\MM_i \in \CC^+ : i \in \omega)$ such that $\MM_i$ contains elements $(a_j^i : j \in [i])$ satisfying $\MM_i \models \phi(a^i_j; a^i_k) \iff j \leq k$.
\end{definition}

It is easy to see that if $\CC$ transduces $\mathcal{LO}$, then it is not monadically stable. The converse follows from \cite[Proposition 4.9]{braunfeld2022existential}.

We now recall some facts about forking in stable and monadically stable theories that we will need. For these facts, we work in a large saturated model $\UU$ (for us, size $\cont$ suffices) of a stable theory $T$, and all of the sets and models we consider have cardinality strictly less than $\UU$ (at most \cont). 

\textbf{Convention:} For the rest of this section, we will assume that the set of constants name an elementary submodel $\NN \prec \UU$. This convention will essentially be in force for the entire paper, since we will work almost exclusively with \emph{grounded} theories (see Definition \ref{def:grounded} and the paragraph thereafter). This means the formulas may freely use elements from $\NN$ as constants, while types are implicitly over $\NN$, and possibly additional explicit parameters.

Furthermore, when we define forking below, this is really \emph{forking over $\NN$}. Even the definition we give, which is nonstandard although equivalent to the standard definition in our setting, makes use of our assumption that the constants name a model. (We are essentially using the characterization that the unique non-forking extension of a type over the model $\NN$ to some $A \supset N$ is its unique extension that is definable over $\NN$. See \cite[Lemma 8.1.5, Corollary 8.3.7]{tent2012course}.) A more general exposition on forking can be found e.g. in \cite[Chapters 8-9]{tent2012course}.

\begin{definition}
    Let $\phi(\xbar; \ybar)$ be a partitioned formula (i.e. we have partitioned the free variables into the two tuples $\xbar$ and $\ybar$). Given $\abar \subset U$, the \emph{$\phi$-type of $\abar$ (over $\NN$)}, denoted $\tp_\phi(\abar)$, is the set of formulas $\set{\phi(\xbar; \nbar) | \nbar \in N^{|\ybar|}, \UU \models \phi(\abar; \nbar)}$.

    The space of $\phi$-types, denoted $S_\phi(T)$, is the set of all $\phi$-types realized in $\UU$.

    We also let $\phi^*(\xbar; \ybar) := \phi(\ybar; \xbar)$. Thus the \emph{$\phi^*$-type of $\bbar$}, denoted $\tp_{\phi^*}(\bbar)$, is $\set{\phi(\nbar; \ybar) | \nbar \in N^{|\xbar|}, \UU \models \phi(\nbar; \bbar)}$, and we also have the corresponding space $S_{\phi^*}(T)$.
\end{definition}

\begin{fact} [{Definability of $\phi$-types, see e.g. \cite[Theorem 8.3.1]{tent2012course}}] \label{fact:deftypes}
    Let $T$ be a stable theory. Given a formula $\phi(\xbar; \ybar)$, and a tuple $\abar \in \UU^{|\xbar|}$, there is a formula $\psi(\ybar)$ that is a boolean combination of formulas of the form $\phi(\nbar_i; \ybar)$ with each $\nbar_i \subset N$, such that for all $\nbar' \subset N$ we have $\UU \models \phi(\abar ; \nbar') \leftrightarrow\psi(\nbar')$.
\end{fact}
\begin{remark}
    Note that the formula $\psi(\ybar)$ depends only on $\tp_\phi(\abar)$, rather than on $\abar$.

    Also, although we will not use this, one can prove that the formula $\psi(\ybar)$ can be chosen to be a ``majority vote'' of instances of $\phi$, i.e. $\psi$ uses constants $\nbar = \nbar_0 \dots \nbar_k$ and $\psi(\ybar)$ says that $\phi(\nbar_i; \ybar)$ holds for the majority of $0 \leq i \leq k$.
\end{remark}

While the formula $\psi(\ybar)$ above determines whether $\phi(\abar; \ybar)$ holds for tuples in $\NN$, one should think that it also generically determines whether $\phi(\abar; \ybar)$ holds even for tuples outside of $\NN$. The notion of forking we now define describes when a formula implies that any tuple satisfying it does not follow this generic behavior. Returning to our earlier examples, a ``generic point'' in an infinite-dimensional vector space will not be in the span of some fixed finite set, and so any formula $\phi(x; \bbar)$ expressing $x$ as a non-zero linear combination of elements from $\bbar$ will fork. Similarly, in a bounded-degree graph (with infinitely many connected components), a ``generic point'' will not be in any of the components of some fixed finite set.

\begin{definition}
  	Let $T$ be a stable theory. Fix a formula $\phi(\xbar; \ybar)$ and a tuple $\bbar \in \UU^{|\ybar|}$. We say $\phi(\xbar; \bbar)$ \emph{forks (over $\NN$)} if there is some formula $\phi'(\bbar; \xbar)$ and corresponding (to $\phi'$) formula $\psi'(\xbar)$ as in Fact \ref{fact:deftypes} such that $\UU \models \phi(\xbar; \bbar) \rightarrow (\phi'(\bbar; \xbar) \leftrightarrow \neg \psi'(\xbar))$.
	
	Note that there may be several possibilities for formula $\psi'(\xbar)$, and implicit in this definition is that it does not depend on which one we choose.
\end{definition}

For example, in a bounded-degree graph, let $b \in \UU\setminus \NN$, and let $\phi(x; b)$ say that $x$ is at distance at most $k$ from $b$ for some fixed $k$. We show that $\phi$ forks. We let $\phi'(b; x)$ be $\phi(x; b)$. Since $\NN \prec \UU$, it must be a union of connected components, and so no $n \in \NN$ satisfies $\phi'$; thus we may take $\psi'(x)$ to be $x \neq x$. The case of vector spaces where $\phi(x; \bbar)$ expresses $x$ as a non-zero linear combination of elements from $\bbar \subset \UU \setminus \NN$ is similar, since $\NN$ must be a subspace and so no $n \in \NN$ will satisfy $\phi$. 

Given a tuple $\abar$, we now say that \emph{$\abar$ forks with $\bbar$} (or that \emph{$\abar$ and $\bbar$ are forking-dependent}) if there is some $\phi(\xbar; \bbar) \in \tp(\abar/\bbar)$ that forks. We write this as $\abar \nind \bbar$, and that $\abar$ does not fork with $\bbar$ (or that \emph{$\abar$ and $\bbar$ are forking-independent}) as $\abar \ind \bbar$. Note that whether $\abar \nind \bbar$ depends only on $\tp(\abar\bbar)$. Finally, given sets $A,B$, we say that \emph{$A$ forks with $B$}, denoted as $A \nind B$, if there are finite tuples $\abar \subset A, \bbar \subset B$ such that $\abar \nind \bbar$. $A \ind B$ is defined analogously. 

More particularly, we might say that $\abar$ forks with $\bbar$ \emph{via the formula $\phi$} if $\phi(\xbar; \bbar) \in \tp(\abar/\bbar)$ forks.

\begin{definition}[Finite satisfiability]
Given a tuple $\abar$ and sets $A \subset B$, we say that $\tp(\abar/B)$ is \emph{finitely satisfiable in $A$} if for every formula $\phi(\xbar)$ (with parameters from $B$) in $\tp(\abar/B)$, there is some $\abar' \subset A$ such that $\UU \models \phi(\abar')$;
\end{definition}

We recall the following facts. Stationarity and the bound on $\phi$-types follow easily from definability of $\phi$-types and our definition of forking. For the other facts, see e.g. \cite[Chapter 8]{tent2012course}.

\begin{fact} \label{fact:forking}
	Let $T$ be a stable theory. The relation $\ind$ satisfies the following properties, for all subsets $A,B$ of $\UU$ and tuples $\abar, \bbar \subset U$:
	\begin{enumerate}[(1)]
		\item (Symmetry:) 
        $A \ind B \iff B \ind A$.
		\item (Finite satisfiability) $\abar \ind B$ if and only if $\tp(\abar/NB)$ is finitely satisfiable in $N$. Thus, if $B \subset N$ then $\abar \ind B$ for every $B \subset U$.
		\item (Stationarity) If $\abar \ind \bbar$, then whether $\UU \models \phi(\abar; \bbar)$ is determined by $\tp_\phi(\abar)$ and $\tp_{\phi^*}(\bbar)$. Thus, if $\abar' \models \tp_\phi(\abar)$ and $\bbar' \models \tp_{\phi^*}(\bbar)$ but $\UU \models \phi(\abar; \bbar) \leftrightarrow \neg \phi(\abar'; \bbar')$, then $\abar' \nind \bbar'$.
		\item (Bound on $\phi$-types) For every formula $\phi(\xbar; \ybar)$, $|S_\phi(T)| \leq \az$.
		\item (Extension) Given $p = \tp(\abar/A)$ and $B \supset A$, there is some $\abar' \models p$ such that $\abar' \ind B$.
	\end{enumerate}
\end{fact}

\begin{fact}[{\cite[Lemma 4.2.6]{BS1985monadic}}] \label{fact:monstab}
	Let $T$ be monadically stable. Then forking has the following additional properties.
	\begin{enumerate}[(1)]
		\item (Total triviality) If $A \nind B$, then there are singletons $a \in A$ and $b \in B$ such that $a \nind b$.
		\item (Transitivity on singletons) If $a \nind b$ and $b \nind c$ then $a \nind c$. 
	\end{enumerate}
\end{fact}

 Facts \ref{fact:forking}(1) and \ref{fact:monstab}(2) imply that in monadically stable theories, forking (over $N$) defines an equivalence relation on the complement of $N$. Furthermore, by Fact \ref{fact:monstab}(1), two sets fork if and only if there is some class of this equivalence relation that they both intersect. Also, Fact \ref{fact:forking}(3) shows that the relations between distinct classes are controlled by how those classes individually relate to $N$.

\section{Grounded theories and Tarski-Vaught types}
\label{sec:Stone}
As discussed in Section \ref{sec:fork}, our presentation is smoother if we work over a model named by constants. We now formalize this in the definition of \emph{grounded theories}, and introduce some important type spaces for a grounded theory.

\begin{definition} \label{def:grounded}
Assume that $\sigma$ is a countable signature and $T$ is a complete theory in the first-order language of $\sigma$-structures.
We say that $T$ is \emph{grounded} if, for every formula of the form $(\exists x)\,\varphi(x)$ in $T$ there exists a constant $c\in\sigma$ such that $\varphi(c)\in T$.    
\end{definition}
Note that this property is equivalent (as a direct consequence of the Tarski--Vaught criterion) to the property that for every model $\mathbf M$ of $T$, the substructure of $\mathbf M$ induced by the constants is an elementary substructure of $\mathbf M$. In other words, if a theory is grounded, then it has a  model $\mathbf N$ whose domain is exactly the set of all the constants.

\begin{definition} Let $T$ be a grounded theory. A \emph{Tarski--Vaught type} (or \emph{TV-type}) of $T$ is a complete type $p\in S_\omega(T)$ that satisfies the following conditions:
\begin{itemize}
	\item for every $i<j<\omega$, $(x_i\neq x_j)\in p$;
	\item for every $i<\omega$ and every constant $c\in\sigma$, $(x_i\neq c)\in p$;
	\item for every formula $\psi(\bar x)\in p$ of the form
	$(\exists y)\ \varphi(\bar x,y)$, either there exists a constant $c\in\sigma$ such that $\varphi(\bar x,c)\in p$, or there exists $i<\omega$ such that $\varphi(\bar x,x_i)\in p$ (or both).
\end{itemize}
We denote by $S_{TV}(T)$ the set of all the TV-types in $S_\omega(T)$ and by $N$ the set of all the constants of $\sigma$.

\end{definition}

\begin{lemma}
	\label{lem:TV}
	$S_{TV}(T)$ is a Borel subset of $S_\omega({T})$.
\end{lemma}
\begin{proof}
	The set $S_{TV}(T)$  is the intersection of the closed set $S_\omega(T)$ with
    
\begin{itemize}
	\item the countably many clopens $\clop(x_i\neq x_j)$ (for   $i<j<\omega$),
	\item the countably many clopens $\clop(x_i\neq c)$ (for  $i<\omega$ and constant $c\in N$),
	\item countably many open sets of the form
	\[\bigcup_{c\in N}\clop(\varphi(\bar x,c))\cup\bigcup_{i<\omega}\clop(\varphi(\bar x,x_i)).\]
\end{itemize}
\end{proof}

\begin{notation} \label{not:proj}
    Fix a theory $T$. We let $\proj \colon S_\omega(T) \to S_1(T)$ be the projection that maps a type $p(x_0, x_1,\dots) \in S_\omega(T)$ to the subset of formulas it contains that have only $x_0$ as a free variable.

\end{notation}
Note that $\proj$ is continuous.

\begin{lemma}
	\label{lem:proj}
	Let $T$ be a grounded theory (with set of constants $N$). Then,
	\[
	\proj(S_{\sf TV}(T))=S_1(T)\setminus\set{\tp(c) |  c\in N}.
	\]
\end{lemma}
\begin{proof}
	First, note that $\proj(S_{\sf TV}(T))\subseteq \proj(S_\omega(T))=S_1(T)$.
	Moreover, for every constant $c\in N$, the formula $(x_0\neq c)$ belongs to all the types in $S_{\sf TV}(T)$. Hence, we have the left-to-right inclusion.

    For the reverse inclusion, fix some $p(x_0) \in S_1(T)\setminus\set{\tp(c) |  c\in N}$. Consider some $\mathbf M \succ \NN$ realizing $p$ such that $M \setminus N$ is countably infinite. Choose any bijection $f:  \set{x_i : i \in \omega} \to M \setminus N$ such that $f(x_0)$ realizes $p$. Then the image of $f$ realizes a type in $S_{\sf TV}(T)$ that projects to $p$.
\end{proof}

\begin{definition} \label{def:NF(T)}
Let $T$ be a grounded stable theory with $\NN$ the model named by constants, and let $k \in \omega$. We let $\NF_k(T) \subset S_k(T)$ be the set of $k$-types realized by elements $(a_0, \dots, a_{k-1})$ in some model such that $a_i \ind a_j$ for every $i \neq j$.

We then let $\pi_k \colon \NF_k(T) \to S_1(T)^k$ be the mapping sending
$\tp(a_0, \dots, a_{k-1})$ to $
(\tp(a_0), \dots, \tp(a_{k-1}))$.
\end{definition}

\begin{lemma} \label{lem:NFk Borel}
	Let $T$ be a grounded stable theory. Then $ \NF_{k}(T)$ is a Borel subset of  $S_k(T)$.
\end{lemma}
\begin{proof}
Fix $i < j < \ell$. It suffices to show the set $X_{i,j,\phi}$ of $k$-types realized by tuples $\abar = (a_0, \dots, a_{k-1})$ such that $a_i$ does not fork with $a_j$ via $\phi(x; y)$ is Borel.

 By Facts \ref{fact:forking}(3)-(4), there is some (countable) $Q_\phi \subset S_\phi(T) \times S_{\phi^*}(T)$ (consisting of those pairs for which non-forking realizations satisfy $\phi$) such that $X_{i,j,\phi}$ consists of those types $p(x_0, \dots, x_{k-1})$ such that: either $p \supset q(\xbar_i) \cup q'(\xbar_j) \cup \set{\phi(\xbar_i; \xbar_j)}$ for some $(q, q') \in Q_\phi$, or $p \supset q(\xbar_i) \cup q'(\xbar_j) \cup \set{\neg \phi(\xbar_i; \xbar_j)}$ for some $(q, q') \notin Q_\phi$.
\end{proof}

\begin{lemma} \label{lem:Boreliso}
	Let $T$ be a grounded monadically stable theory. The mapping $\pi_{k} \colon \NF_k(T) \to S_1(T)^k$ is a continuous Borel isomorphism.
\end{lemma}
\begin{proof}
    That $\pi_{k}$ is continuous is trivial.
	
	In order to prove the surjectivity, we consider a sufficiently saturated model $\mathbf M$.
	Let $(p_0,\dots,p_{k-1})\in S_{1}(T)^k$. By Fact \ref{fact:forking}(5), there exist $a_0,\dots,a_{k-1} \in M$ with
	$ a_i\ind \set{a_0, \dots, a_{i-1}}$ for all $0 < i<k$ and $\tp(a_i)=p_i$ for all $0 \leq i<k$. Thus, 
	$\tp(a_0, \dots, a_{k-1})\in  \NF_{k}(T)$ and $\pi_{k}(\tp(a_0, \dots, a_{k-1}))=(p_0,\dots,p_{k-1})$.
	
	Since $T$ is monadically stable, Fact \ref{fact:monstab}(1) implies that $a_i \ind \set{a_0, \dots, a_{k-1}}\setminus\set{a_i}$ for every $0 \leq i < k$. The injectivity of $\pi_{k}$ then follows from Fact \ref{fact:forking}(3).
	Hence, $\pi_{k}$ is a bijection. As $ \NF_{k}(T)$ and  $S_{1}(T)^k$ are standard Borel spaces, $\pi_{k}$ is a Borel isomorphism by Theorem~\ref{fact:Suslin}.
\end{proof}

\section{Totally Borel models for monadically stable theories}
\label{sec:TB}

In this section, we consider a 	 countable signature $\sigma$, a grounded monadically stable theory $T$ of $\sigma$-structures, and a non-empty Borel subset $S$ of $S_{\sf TV}(T)$. As usual, we fix a suitably saturated model $\UU$, so the set of constants names an elementary substructure $\NN \prec \UU$.

Let $\bI = [0,1] \subset \mathbb R$. We aim to prove the following theorem.

\begin{theorem} \label{thm:totBorel}
There is $\tb \prec \UU$ that can be equipped with a topology so that it is a totally Borel model whose domain can be identified with $(\bI \times S \times \omega) \sqcup N$, satisfying the following conditions.
\begin{enumerate}
    \item For every $r \in \bI$ and $p \in S$, $B(p, r) \models p$, where $B(p, r):=((r, p, i) : i \in \omega)$;
    \item The $\omega$-tuples in $\set{B(p, r) | p \in S, r \in \bI}$ are pairwise forking-independent.
\end{enumerate}
\end{theorem}    

We remark that there is a Borel completeness theorem, which produces a totally Borel model of size continuum for any countable first-order theory with an infinite model. However, the proofs (\cite[Theorem 2.1]{montalban2013borel}, \cite[Theorem 3.1.1]{steinhorn1985chapter}) seem to allow for realizing only countably many types. And there are non-trivial limits (i.e. beyond just the constraint that models have size $\cont$) to the amount of saturation that can be achieved in general in totally Borel models \cite[Theorem 1.3.3]{steinhorn1985chapter}, so in the Borel setting one cannot first build a model realizing countably many types and then freely realize types over it afterwards.

As discussed in Section \ref{sec:sketch}, a crucial point for us is roughly that there should be a Borel $A \subset \tb$ realizing all (possibly uncountably many) 1-types over $\NN$. The measure on $\tb$ will eventually be obtained by pulling back the measure from $S_1(\NN)$ to $A$.

We will need the following general result about building models of monadically stable theories, saying that putting together non-forking elementary submodels of $\UU$ (over a common elementary submodel $\NN$) again yields an elementary submodel.

\begin{lemma} \label{lem:elsub}
Let $ \UU \supset M = \bigcup_{i \in \kappa} M_i$ be such that $\NN \prec\MM_i \prec \UU$ for all $i \in \kappa$, and $M_i \ind M_j$ for all $i \neq j \in \kappa$. Then $\MM \prec \UU$.
	\end{lemma}
\begin{proof}
	Consider a tuple $\mbar \subset M$ and a formula $\phi(\xbar, y)$ such that $\UU \models \exists y \phi(\mbar, y)$, and let $u \in \UU$ witness the existential. We may assume that $\mbar \cap N = \emptyset$, since otherwise we may add elements from $N$ as parameters to $\phi$ and shrink $\mbar$. Let $I = \set{i \in \kappa | \mbar \cap M_i \neq \emptyset}$, and let $\mbar_i = \mbar \cap M_i$ for every $i \in I$. 
	
	First, suppose $u \ind \mbar$. By Fact \ref{fact:forking}(2), $\tp(u/\mbar)$ is finitely satisfiable in $N$, and so there is some $n \in N$ such that $\UU \models \phi(\mbar, n)$.
	
	So suppose that $u \nind \mbar$. By Fact \ref{fact:monstab}(1), $u \nind M_{i^*}$ for some $i^* \in I$. Let $\mbar'$ be such that $\mbar = \mbar'\mbar_{i^*}$ (up to permutation). By Fact \ref{fact:monstab}(1)-(2), $u\mbar_{i^*} \ind \mbar'$; indeed, by the comments after Fact \ref{fact:monstab}, forking yields an equivalence relation on singletons off $\NN$, and two sets fork if and only if they intersect the same equivalence class. By assumption, $\mbar'$ does not intersect an equivalence class of $\mbar_{i^*}$, and $u$ must lie in the same equivalence class as some element of $\mbar_{i^*}$.
    
    Let $\phi'(\xbar'; \zbar, y) = \phi(\xbar'\zbar, y)$, where the partition $\xbar'\zbar$ corresponds to the partition $\mbar'\mbar_{i^*}$. By our definition of (non-)forking, there is some formula $\psi(y, \zbar)$ (obtained from definability of $\tp_{\phi'}(\mbar')$) such that if $a\bbar \ind \mbar'$, we have $\UU \models \psi(a, \bbar)$ if and only if $\UU \models \phi(\mbar'\bbar, a)$. We have $\UU \models \phi(\mbar'\mbar_{i^*}, u)$, thus $\UU \models \psi(u, \mbar_{i^*})$, and thus $\UU \models \exists y \psi(y, \mbar_{i^*})$. Since $\MM_{i^*} \prec \UU$, we have $\MM_{i^*} \models \exists y \psi(y, \mbar_{i^*})$. If $v$ is a witness for this, then we will have $\UU \models \phi(\mbar'\mbar_{i^*}, v)$, as desired.
	\end{proof}

 By applying Fact \ref{fact:forking}(5), we may find $\mathbf{B} \subset \UU$ consisting of \cont-many realizations $\set{\abar_{p, i} | i \in \cont}$ of every type $p \in S \subset S_{TV}(T)$,  
 and such that $\abar_{p,i} \ind \abar_{q, j}$ if $(p,i) \neq (q, j)$, and thus $\abar_{p,i} \ind \set{\abar_{q,j} | (q, j) \neq (p, i)}$. Thus by Fact \ref{fact:forking}(3), the isomorphism type of $\tb$ is uniquely determined. We identify the domain of $\mathbf{B}$ with $(\bI \times S \times \omega) \sqcup N$ in the natural way, i.e. so that given $\alpha \in \bI$ and $p \in S$, we have $((\alpha, p,  0), (\alpha,p,1), \dots) \models p$.
In the following, for $p\in S$ and $\alpha \in \bI$, we denote by $B(p, \alpha)$ the set $\set{\alpha}\times \set{p} \times \omega$ (cf Fig.~\ref{fig:TB}).

\begin{figure}[h!t]
	\begin{center}
		\includegraphics[width=.9\textwidth]{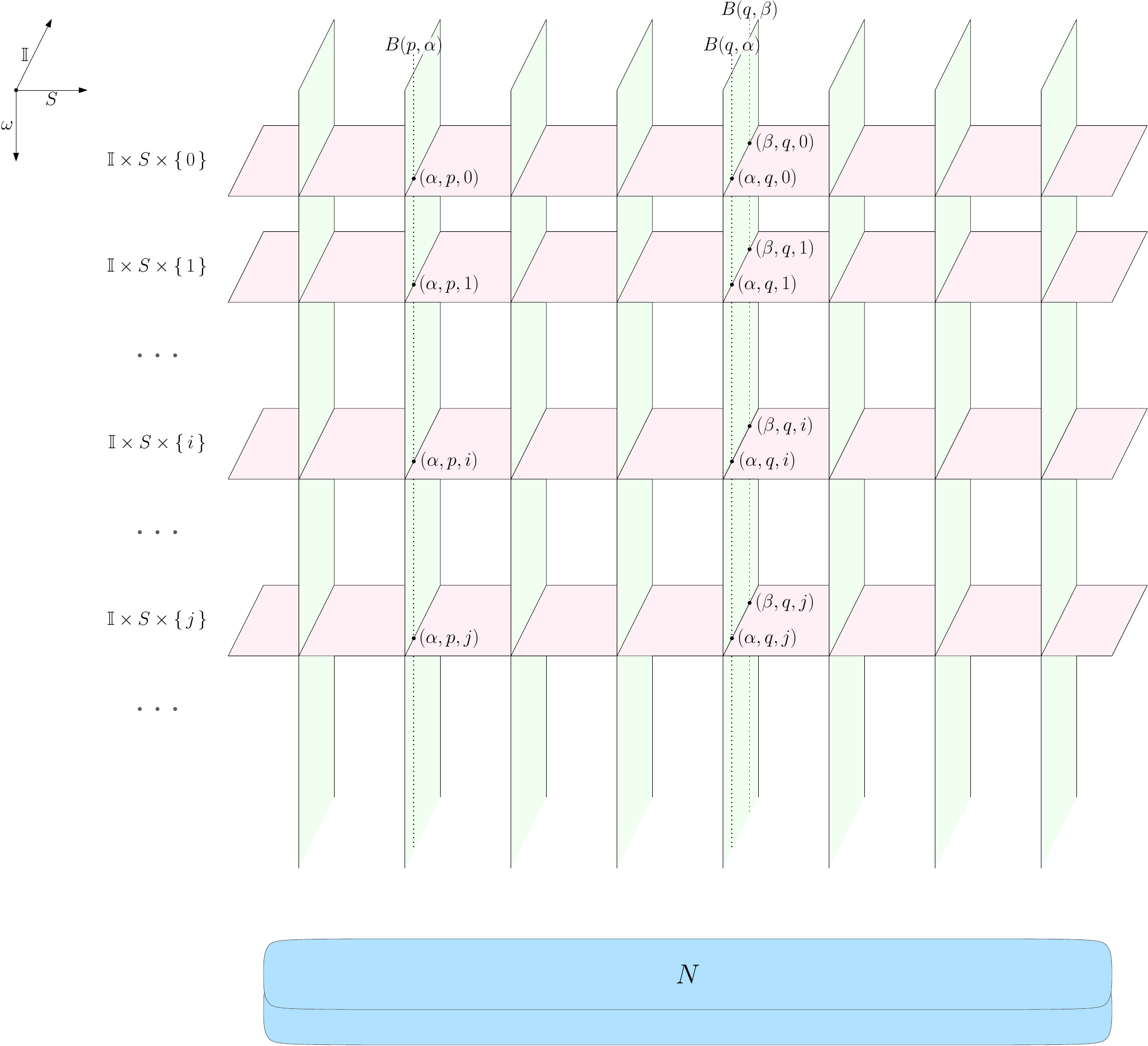}
	\end{center}
	\caption{Structure of the totally Borel model $\tb$. 
    In the model, we have $(\alpha,p,i)\ind(\beta,p,j)$ if $(\alpha,p)\neq(\beta,q)$.}
	\label{fig:TB}
\end{figure}

Next, we equip $\mathbf{B}$ with the topology obtained by taking the discrete topology on $N$, the standard topologies on $\bI$, the induced topology on $S$ (note that the induced topology on $S$ need not be Polish, but $S$ with the corresponding Borel $\sigma$-algebra will still be a standard Borel space), the discrete topology on $\omega$, and the product topology on $\bI \times S \times \omega$.

\begin{theorem}
	The structure $\tb$ is a totally Borel model of $T$, satisfying the conditions of Theorem \ref{thm:totBorel}.
\end{theorem}
\begin{proof}

    That $\tb \models T$ is immediate from Lemma \ref{lem:elsub}, and it clearly satisfies the two conditions of Theorem \ref{thm:totBorel}. So it remains to prove $\tb$ is totally Borel.

        We consider the case of formulas with at most three free variables.
    The case of higher-arity formulas follows exactly the same argument but with more cumbersome notation.

    Let $\psi$ be a formula with a single free variable. Then, $\psi(\tb)$ is Borel as
\[
\psi(\tb)=\psi(N)\cup\bigcup_{i<\omega} (\bI\times \set{p\in S| \psi(x_i)\in p}\times\{i\}),
\]    
$\psi(N)$ is countable (hence Borel), and $\set{p\in S| \psi(x_i)\in p}$ is clopen.
\medskip

We now consider a formula $\varphi(x,y)$ with two free variables. 
    It suffices to show that $\varphi(\mathbf B \backslash N)$ is Borel, since 
    \[\varphi(\mathbf B) = \varphi(\tb \backslash N) \cup \bigcup_{n\in N}(\varphi(\tb, n)\times\set{n} \cup \set{n}\times\varphi(n, \tb)),\] 
 and the sets $\varphi(\tb,n)$ and $\varphi(n,\tb)$ are Borel (by the above one free variable case).

    Let  $\tb_{\set{x,y}}$ and $\tb_{\set{x}, \set{y}}$ be the subsets of  $(\tb \backslash N)^2$ given by 
    \begin{align*}
    \tb_{\set{x,y}}&=\quad\ \bigcup_{\mathclap{(p,r) \in S\times \bI}}\quad (B(p, r) \times B(p, r))\\
    \tb_{\set{x}, \set{y}}&=\bigcup_{\mathclap{(p,r) \neq (q, r') \in S \times \bI}} (B(p,r) \times B(q, r')).    
    \end{align*}

Note these sets partition $(\tb \backslash N)^2$, and both are Borel, essentially because the diagonal of $(S\times \bI)^2$ is closed.

  We first show that $\varphi(\tb_{\set{x,y}})$ is Borel. Let $S_{i,j} = \set{p \in S | \varphi(x_i, x_j) \in p}$, which is clopen. Then \[\varphi(\tb_{\set{x,y}}) = \tb_{\set{x,y}}\cap \bigcup_{i, j \in \omega} \bigl((\bI\times S_{i,j}  \times \set{i}) \times (\bI\times S_{i,j}  \times \set{j})\bigr).\]

  We next show that $\varphi(\tb_{\set{x},\set{y}})$ is Borel. Given $q(x_i) \in S_\varphi(T)\cup S_{\varphi^*}(T)$, let $S_{q(x_i)} = \set{p \in S | p \supset q}$, which is closed. By Fact \ref{fact:forking}(3), there is some $Q \subset S_\varphi(T) \times S_{\varphi^*}(T)$ (which is countable by Fact \ref{fact:forking}(4)) such that 
  \[\varphi(\tb_{\set{x},\set{y}}) = 
   \tb_{\set{x},\set{y}}\cap  \bigcup_{\substack{i, j \in \omega\\ \mathclap{(q_1(x_i), q_2(x_j)) \in Q}}}
  \bigl((\bI\times S_{q_1(x_i)}  \times \set{i}) \times (\bI\times S_{q_2(x_j)} \times \set{j})\bigr).\]

 In the higher-arity case, we would have to use induction on formulas with fewer free variables to eliminate the case where some free variables are interpreted as constants, then consider
 all possible partitions of the free variables. 
 
 For example, for the case of a ternary formula $\psi(x,y,z)$, one would consider the sets 
     \begin{align*}
  \tb_{\set{x,y, z}}&=\qquad\ \bigcup_{\mathclap{(p,r) \in S\times \bI}}\quad (B(p, r) \times B(p, r) \times B(p,r)),\\
  \tb_{\set{x, y}, \set{z}}&=\quad \bigcup_{\mathclap{(p,r) \neq (q, r') \in S \times \bI}} (B(p,r) \times B(p,r) \times B(q, r')),\\
  \tb_{\set{x}, \set{y, z}}&=\quad\bigcup_{\mathclap{(p,r) \neq (q, r') \in S \times \bI}}  (B(p,r) \times B(q,r') \times B(q, r')),\\
  \tb_{\set{x}, \set{y}, \set{z}}&=\bigcup_{\substack{\mathclap{(p,r), (q, r'), (s, r'') \in S \times \bI}\\ \text{pairwise distinct}}} (B(p,r) \times B(q,r') \times B(s, r'')),
    \end{align*}
    and show $\psi$ defines a Borel subset of each. 
    
    The case of $\psi(\tb_{\set{x,y, z}})$ is essentially the same as that of $\varphi(\tb_{\set{x,y}})$, and $\psi(\tb_{\set{x}, \set{y}, \set{z}})$ is essentially the same as $\varphi(\tb_{\set{x},\set{y}})$. We now argue that $\psi(\tb_{\set{x}, \set{y, z}})$ is Borel, and the argument for $\psi(\tb_{\set{x, y}, \set{z}})$ is the same.

    Consider $\psi$ with the partition $\psi(x; y,z)$. Given $q(x_i) \in S_\psi(T)$, let $S_{q(x_i)} = \set{p \in S | p \supset q}$, and given $q(x_i,x_j) \in S_{\psi^*}(T)$, let $S_{q(x_i, x_j)} = \set{p \in S | p \supset q}$. As in the case of $\varphi(\tb_{\set{x},\set{y}})$, there is some countable $Q \subset S_\psi(T) \times S_{\psi^*}(T)$ such that
    \begin{align*}
    &\psi(\tb_{\set{x},\set{y, z}}) = 
   \tb_{\set{x},\set{y, z}}\cap \bigcup_{\substack{i, j, k \in \omega\\ \mathclap{(q_1(x_i), q_2(x_j, x_k)) \in Q}}}Z_{i,j,k},\\
   \intertext{where}
   Z_{i,j,k}&=
  \bigl((\bI\times S_{q_1(x_i)}  \times \set{i}) \times (\bI\times S_{q_2(x_j, x_k)} \times \set{j}) \times (\bI\times S_{q_2(x_j, x_k)} \times \set{k})\bigr).
    \end{align*}
\end{proof}

We note that in the case where $S = S_{\sf TV}(T)$ and so the induced topology on $S$ is Polish, the definable relations of $\tb$ are not only Borel, but are in the bottom few levels of the Borel hierarchy.

\section{Dismantling models}
\label{sec:dismantle}
In this section, we prove an infinitary analogue of the result characterizing nowhere dense classes via quasi-residuality in \cite{modeling_jsl}, generalized to the monadically stable setting. We recall that quasi-residuality roughly says that in every graph in the class, there is a set $N$ of vertices of bounded size such that the rest of the graph decomposes into small disconnected pieces over $N$.

\begin{definition}
    Let $T$ be a monadically stable theory. Let $\MM \models T$ be equipped with a probability measure on the sigma-algebra generated by parameter-definable sets, and let $\NN \prec \MM$. Then we say that $\MM$ is \emph{dismantled over $\NN$} if, after naming $N$ by constants, the equivalence classes of the forking relation on singletons have measure 0.

    In the case where $T$ is grounded, we say that $\MM$ is \emph{dismantled} if it is dismantled over the set of constants.
\end{definition}

We now prove that we may always choose some $N$ to dismantle models. Beyond the monadically stable setting, similar results should hold for measures such as frequency interpretation measures \cite[Definition 7.23]{simon2015guide}, where one can find a small model over which they are invariant.

\begin{lemma} \label{lem:0petals}
Let $T$ be monadically stable in a countable language $\sigma$. Let $\MM \models T$ be equipped with a probability measure $\mu$ on the sigma-algebra generated by parameter-definable sets (in one dimension). Then there is some countable $\NN \prec \MM$ such that (after naming every element of $N$ by a constant) for all $m \in M$, we have $P_m = \set{x \in M | x \nind m}$ is measurable and null.
\end{lemma}

\begin{remark}
Note that not every $\NN \prec \MM$ will work. A simple example is that if some singleton in $M$ has positive measure then it must be included in $N$. Less trivially, $\MM$ might be an equivalence relation with infinitely many infinite classes, and some class might have positive measure even though no singleton does. If $N$ contains no point from this class, then the class will also be a class of the forking equivalence relation. But if $N$ contains some point from the class, then over $N$ each singleton from the class will form its own class of the forking equivalence relation, which will have measure 0.
\end{remark}

We will use that an analogue of Fact \ref{fact:forking}(5) holds for finite satisfiability in a model, even in unstable theories.

\begin{fact}[{\cite[Lemma 8.1.3]{tent2012course}}] \label{fact:fs ext}
	Let $T$ be an arbitrary theory, and let $\NN \prec \UU \models T$ with $\UU$ sufficiently saturated. Let $a \in \UU$ and $B \supset A \supset N$.
	
	(1) $\tp(a/N)$ is finitely satisfiable in $N$.
	
	(2) If $\tp(a/A)$ is finitely satisfiable in $N$, then there is some $a' \in \UU$ such that $a' \models \tp(a/A)$ and $\tp(a'/B)$ is finitely satisfiable in $N$.
\end{fact}

\begin{proof}[Proof of Lemma \ref{lem:0petals}]
Expand the language to $\sigma'$, adding relations $R_{\phi(x; \ybar), c}(\ybar)$ for every $\sigma$-formula $\phi(x; \ybar)$ (without parameters) and every $c \in \mathbb{Q} \cap [0,1]$, and expand $\MM$ to $\MM'$ by interpreting $R_{\phi(x; \ybar), c}(\mbar)$ as ``$\mu(\phi(x; \mbar)) \geq c$'' (note $\MM'$ may no longer be stable). Let $\NN' \prec \MM'$ be countable, and let $\NN$ be the $\sigma$-reduct of $\NN'$. Then we claim $\NN$ is as desired. 

Recall that we have only defined forking when we have named an elementary substructure by constants. For the rest of this proof, we let $\sigma_N$ be the expansion of $\sigma$ by constants for $N$, and say that an $\sigma_N$-formula $\phi(x, \mbar)$ forks over $N$ if it forks in the expansion where we have named $N$ by constants.
	
		For every $m \in M$, we have \[P_m = \bigcup \set{\phi(M, m) | \phi(x, y) \text{ an $\sigma_N$-formula, $\phi(x, m)$ forks over $N$}}.\] So $P_m$ is a countable union of parameter-definable sets, and thus is measurable.
	
	Now suppose there is some $m \in M$ such that $\mu(P_m) > 0$. Then by countable additivity there is some $\sigma_N$-formula $\phi(x, m)$ that forks over $N$ such that $\mu(\phi(x, m)) > 0$. Let $1/r \in \mathbb{Q}$ be such that $\mu(\phi(x, m)) \geq 1/r$. Note that $m \not\in N$, since then $P_m = \emptyset$.
	
	Consider the $\sigma'$-expansions $\NN' \prec \MM'$. Let $\phi(x, y)$ be $\psi(x, y, \nbar)$ where $\psi$ is without parameters from $N$. Note that it is part of ${\rm Th}(\MM')$ that there do not exist elements $(m_i : i \in [r+1])$ such that the sets defined by $\set{\phi(x, m_i) | i \in [r+1]}$ are pairwise disjoint and $R_{\psi(x; y \zbar), 1/r}(m_i\nbar)$ holds for every $i \in [r+1]$.
	
	Now let $\UU' \succ \MM'$ be $\aleph_1$-saturated, and let $\UU$ be its $\sigma$-reduct. Then (by iterating Fact \ref{fact:fs ext}) $U'$ contains a sequence $(m_i : i \in \omega)$ such that  $\tp(m_i/N') = \tp(m/N')$ for all $i \in \omega$ and such that $\tp(m_i/N'm_0 \dots m_{i-1})$ is finitely satisfiable in $N'$ for every $i \in \omega$. In particular $\UU' \models R_{\psi(x; y \zbar), 1/r}(m_i\nbar)$ for every $i \in \omega$. So we will obtain a contradiction if $\set{\phi(U', m_i) | i \in \omega}$ are pairwise disjoint. Working in $\UU$, where finite satisfiability over $N$ and non-forking over $N$ coincide (Fact \ref{fact:forking}(2)), we see that $m_i \ind m_j$ for all $i, j \in \omega$. Since forking is transitive on singletons in $\UU$ (Fact \ref{fact:monstab}(2)), if there were some $u \in U$ such that $\UU \models \phi(u, m_i) \wedge \phi(u, m_j)$ for $j \neq i$, then we would have $m_i \nind m_j$. So $\set{\phi(U, m_i) | i \in \omega}$ are pairwise disjoint.
\end{proof}

\section{Transferring the measure}
Finally, in this section we show how to place the desired measure on the totally Borel model $\tb$ we have constructed, and prove our main theorems. We continue to assume a countable language throughout.

\label{sec:transfert}

We recall the definition of a measurable structure from Section~\ref{sec:mes struc}, and will use \emph{measurable model} to denote a measurable structure that is a model of a particular theory.

\begin{lemma}
\label{lem:null}
Let $\mathbf M$ be a dismantled measurable model of a grounded monadically stable theory $T$.
Let $k\geq 2$.
The set $\set{\bar u\in \mathbf M^k |  u_i\nind u_j\text{ for some }i<j<k}$ is measurable and has zero measure.
\end{lemma}
\begin{proof}
The argument for measurability is essentially the same as the proof of Lemma \ref{lem:NFk Borel}.

For measure zero, we first consider the case $k=2$. Given $m \in M$, we let $P_m = \set{x \in M | x \nind m}$. From being dismantled and the Fubini property, we have
\begin{align*}
    m_2(\set{(u,v) |  u\nind v})&=\int P_x(y)\,{\rm d}m_2(x,y)\\
    &=\int\int P_x(y)\,{\rm d}m_1(x)\,{\rm d}m_1(y)\\
    &=0.
\end{align*}
The general case easily follows, as 
$\set{\bar u\in \mathbf M^k |  u_a\nind u_b}$ has
the same measure as the set 
$\set{\bar u\in \mathbf M^2 |  u_0\nind u_1}\times M^{k-2}$, that is $0$.
(Here, we use the Fubini property to sort the coordinates.)
\end{proof}

Recall from Section \ref{sec:mod th} the function $\tp^k \colon \MM^k \to S_k(T)$. If $\MM$ is a measurable model, we will want to consider the pushforward $\tp^k_*(m_k)$ of the measure on $M^k$ to a measure on $S_k(T)$.

\begin{lemma}
\label{lem:meas_pres}
Let $\mathbf M$ be a dismantled measurable model of a grounded monadically stable theory $T$. Define the probability measures $\mu_1$ and $\mu_k$ on $S_1(T)$ and $S_k(T)$ by $\mu_1=\tp^1_*(m_1)$ and $\mu_k=\tp^k_*(m_k)$.

Then, $\pi_k:   (\NF_k(T), \mu_k \upharpoonright\NF_k(T)) \to (S_1(T)^k, \mu_1^{\otimes k})$ (recalling Definition \ref{def:NF(T)}) is a probability measure preserving Borel isomorphism.
\end{lemma}
\begin{proof}
    Note that $\mu_k \upharpoonright\NF_k(T)$ is a probability measure, by Lemma \ref{lem:null}. By Lemma \ref{lem:Boreliso}, $\pi_k$ is a Borel isomorphism, so we only have to check it is measure preserving. As these are probability measures, it suffices to check the they agree on clopen sets. Consider the subset $K=\clop(\phi_0)\times\dots\times\clop(\phi_{k-1})$ of $S_1(T)^k$.
    Then, using Lemma \ref{lem:null} and the Fubini property, we have
    \begin{align*}
    \mu_k(\pi_k^{-1}(K))&=m_k\biggl(\Set{\bar x| x_i\ind x_j\ (\forall i<j<k)\text{ and }\mathbf M\models\bigwedge \phi_i(x_i)}\biggr)\\ 
    &=m_k\biggl(\Set{\bar x| \mathbf M\models\bigwedge \phi_i(x_i)}\biggr)\\ 
    &=\prod_{i<k}m_1(\phi_i(\mathbf M))\\
    &=\prod_{i<k}\mu_1(\clop(\phi_i))\\
    &=\mu_1^{\otimes k}(K).
    \end{align*} 
\end{proof}

\begin{lemma}
Let $T$ be a grounded monadically stable theory.
Then, for every dismantled measurable model $\mathbf M$ of $T$ and every formula $\varphi(\bar x)$ with $k$ free variables we have
\[
m_k(\varphi(\mathbf M))=\idotsint I_\varphi(\pi_k^{-1}(t_0,\dots,t_{k-1}))\,{\rm d}\mu_1(t_0)\dots{\rm d}\mu_1(t_{k-1}),
\]
where $I_\varphi$ is the indicator function of $\clop(\varphi)$ and $\mu_1=\tp^1_*(m_1)$.
\end{lemma}
\begin{proof}
Let $\mu_k=\tp^k_*(m_k)$.
According to Lemmas~\ref{lem:null} and~\ref{lem:meas_pres}, we have
\begin{align*}
m_k(\varphi(\mathbf M))&=m_k\bigl(\varphi(\mathbf M)\cap\set{\bar v\in M^k |  v_i\ind v_j\ \forall i<j<k}\bigr)\\
&=\int_{\NF_k(T)} I_\varphi(p)\,{\rm d}\mu_k(p)\\
&=\idotsint I_\varphi(\pi_k^{-1}(t_0,\dots,t_{k-1}))\,{\rm d}\mu_1(t_0)\dots{\rm d}\mu_1(t_{k-1}).
\end{align*}
\end{proof}
\begin{corollary}
	\label{cor:1tok}
Let $\mathbf M$ and $\mathbf M'$ be dismantled measurable models of a grounded monadically stable theory $T$.
If $\tp^1_*(m_1)=\tp^1_*(m_1')$, then
$\tp^k_*(m_k)=\tp^k_*(m_k')$ for all $k<\omega$.
\end{corollary}

We are now ready to prove Theorem~\ref{thm:tb}. In the statement, in order to emphasize that the $\sigma$-algebra on $\tb$ is not inherited from $\MM$, we describe $\tb$ as a separate structure that elementarily embeds into $\MM$, rather than as an elementary substructure. 
\ThmTB*
\begin{proof}
	By Lemma~\ref{lem:0petals}, $\mathbf M$ has a countable elementary substructure $\NN$ such that, if we expand by constants for the elements of $N$, the theory $T$ is grounded and the model $\mathbf M$ is dismantled. So we expand $\MM$ by constants for $N$ and replace $T$ by the theory of the resulting structure. 
    
	Let $\mu_1=\tp^1_*(m_1)$, where $m_1$ is the measure on $\mathbf M^1$.
	Let $\supp(\mu_1)$ be the support of $\mu_1$, that is  the largest closed subset of $S_1(T)$  for which every open neighborhood of every point of the set has positive $\mu_1$-measure.
    Let $\widehat N=\set{\tp(c) |  c\in N}$.
	Recall the function $\proj \colon S_\omega \to S_1$ from Notation \ref{not:proj}. The set $(\proj)^{-1}(\supp(\mu_1)\setminus \widehat N)$ is a closed subset of $S_\omega(T)$. We define
	$S:=S_{\sf TV}(T)\cap (\proj)^{-1}(\supp(\mu_1)\setminus \widehat N)$. (The use of $\supp(\mu_1)$ here rather than all of $S_1(T)$ is not needed for this theorem, but is crucial for Theorem \ref{thm:removal}.)  By Lemma~\ref{lem:TV},  $S_{\sf TV}(T)$ is Borel, and so
	the set $S$ is a Borel subset of $S_\omega(T)$. Note that, according to Lemma~\ref{lem:proj}, $\proj(S)=\proj(S_{\sf TV})\cap (\supp(\mu_1)\setminus \widehat N)=\supp(\mu_1)\setminus\set{\tp(c) |  c\in N}=\supp(\mu_1)\setminus \widehat N$.

	Let $\tb$ be as in Theorem \ref{thm:totBorel}, and fix an elementary embedding $\zeta:\tb\rightarrow\mathbf M$. 
    The map $\proj$
    induces a Borel surjection from $S$ onto 
     $\supp(\mu_1)\setminus \widehat N$ (which are both standard Borel spaces). Thus, we can apply Theorem \ref{thm:Bog} to obtain a measure $\nu_0$ on $S$ such that $(\proj)_*(\nu_0)=\mu_1\upharpoonright (\supp(\mu_1)\setminus \widehat N)$. 
    Letting $\lambda$ be the usual Borel measure on $\bI$ and $\delta$ be the probability measure on $\omega$ concentrated on $\set{0}$, we extend $\nu_0$ to a measure on $\nu$ on $\bI\times S\times \omega\sqcup N$ (that is, to $\tb$) by assigning 
    \[\nu(X) = \lambda\otimes\nu_0\otimes\delta(X\setminus N)+\mu_1(\set{\tp(c) |  c\in X\cap N}).\]
    (Recall that  $X\cap N$, being countable, is a Borel set of $\tb$.)

    Note that considering the product measure of $\lambda,\nu_0$, and $\delta$ on $\bI\times S\times\omega$ ensures that no singleton off $\NN$ will have a positive measure. 
    It follows 
 that $\tb$ is dismantled, as $u\nind v$ implies that $u$ and $v$ belong to the same countable tuple $B(p, r)$.

 We now show that
    $\tp^1_*(\nu)= \mu_1 = \tp^1_*(m_1)$. 
    For $X \subset S_1(T)$, we have
\begin{align*}
    (\tp^1_*(\nu))(X)&=
    \lambda\otimes \nu_0\otimes \delta\bigl(\set{(\alpha,p,i)\in \bI\times S\times\omega|\tp((\alpha,p,i))\in X}\bigr)+\mu_1(X\cap \widehat N).
        \intertext{Hence, as $\delta$ is concentrated on $\set{0}$ and as $\tp((\alpha,p,0))=\proj(p)$ for all $\alpha\in \bI$ and $p\in S$, we have}
    (\tp^1_*(\nu))(X)&=
    \lambda\otimes \nu_0\otimes \delta\bigl(\bI\times\set{p\in S| \proj(p)\in X}\times\set{0}\bigr)+\mu_1(X\cap \widehat N).
\intertext{Thus, as $\proj$ is a Borel surjection from $S$ to $\supp(\mu_1)\setminus \widehat N$ with pushforward $\mu_1\upharpoonright (\supp(\mu_1)\setminus \widehat N)$, and $\lambda(\bI) = 1 = \delta(\set{0})$, we have}
    (\tp^1_*(\nu))(X)&=
    (\proj)_*(\nu_0)(X\cap (\supp(\mu_1)\setminus \widehat N))+\mu_1(X\cap \widehat N)\\
    &=\mu_1(X\setminus \widehat N)+\mu_1(X\cap \widehat N)\\
    &=\mu_1(X).
\end{align*}   
    Hence,  by Corollary \ref{cor:1tok}, 
    we have $\tp_*(\nu^{\otimes k})=\tp_*(m_k)$ for every integer $k$. In other words, for every formula $\psi$ without parameters and with $k$ free variables, we have $\nu^{\otimes k}(\psi(\mathbf M))=m_k(\psi(\mathbf M))$.

    We now consider a formula $\varphi(\bar x,\bar y)$ with $k+p$ free variables and a tuple $\bar b\in (\tb\setminus N)^p$ of parameters (indeed, if a parameter belongs to $N$, we can replace it by the corresponding constant). 
    If $\cbar \in \tb^k$ and $\cbar \ind \bbar$, then by our definition of (non-)forking, there is a single parameter-free formula $\psi(\xbar)$ (obtained from definability of $\tp_{\varphi^*}(\bbar)$) such that $\tb \models \varphi(\cbar, \bbar) \iff \psi(\cbar)$. This yields
    \begin{align*}
    \varphi(\tb,\bar b)\cap\set{\bar x\in \tb^k| \bar x\ind \bar b}&=\psi(\tb) \cap\set{\bar x\in \tb^k| \bar x\ind \bar b}\\
    \varphi(\mathbf M,\zeta(\bar b))\cap\set{\bar x\in \mathbf M^k| \bar x\ind \zeta(\bar b)}&=\psi(\MM)\cap\set{\bar x\in \mathbf M^k| \bar x\ind \zeta(\bar b)},
    \end{align*}
    where the second line follows from the fact that $\zeta$ is an elementary embedding.

  As both $\tb$ and $\mathbf M$ are dismantled, we deduce that 
    \begin{align*}
    \nu^{\otimes k}(\varphi(\tb,\bar b))&=
    \nu^{\otimes k}\biggl(\varphi(\tb,\bar b)\cap\set{\bar x\in \tb^k| \bar x\ind \bar b}\biggr)\\
    &=\nu^{\otimes k}\biggl(\psi(\tb)\cap\set{\bar x\in \tb^k| \bar x\ind \bar b}\biggr)\\
    &=\nu^{\otimes k}\biggl(\psi(\tb)\biggr)\\
    &=m_k\biggl(\psi(\MM) \biggr)\\
    &=m_k\biggl(\psi(\MM)\cap\set{\bar x\in \mathbf M^k| \bar x\ind \zeta(\bar b)}\biggr)\\
    &=m_k\biggl(\varphi(\mathbf M,\zeta(\bar b))\cap \set{\bar x\in \mathbf M^k| \bar x\ind \zeta(\bar b)}\biggr)\\
    &=m_k(\varphi(\mathbf M,\zeta(\bar b))).
    \end{align*}
 \end{proof}

We now prove our main result on the existence of modeling limits. We recall Definition \ref{def:stable} stating that a class $\CC$ is monadically stable if its common theory is.
\ThmMain*
\begin{proof}
Let $(\mathbf M_n)_{n\in\mathbb N}$ be an {\sf FO}-convergent sequence of finite structures from $\mathscr C$. We wish to take the ultraproduct of the sequence equipped with the Loeb measure, and then apply Theorem~\ref{thm:tb}. However, that theorem requires a $\cont$-saturated model, so we will re-index the sequence by $\cont$ in such a way that ultralimits along the re-indexed sequence match limits along the original sequence, and choose our ultrafilter to ensure $\cont$-saturation.

By \cite[Theorem 6.1.4]{chang1990model}, there exists a good, countably incomplete ultrafilter $F$ on $\cont$. (We will not directly need the definition of goodness. Countable incompleteness means that there is a decreasing chain of sets $(A_i \in F : i \in \omega)$ with empty intersection.) Let $(A_i \in F : i \in \omega)$ be a decreasing chain of sets witnessing that $F$ is countably incomplete, with $A_0 = \cont$. Let $\pi : \cont \to \omega$ be given by $\pi(i) = \max\set{j \in \omega | i \in A_j}$, so that $\set{i | \pi(i) \geq n}= A_n \in F$. Now consider the sequence $(\mathbf M_{\pi(i)} : i \in \cont)$ and let $\UU$ be its ultraproduct with respect to $F$. Since $F$ is good, $\UU$ is $\cont$-saturated \cite[Theorem 6.1.8]{chang1990model}.

	For a given $k \in \mathbb{N}$, we denote by $m_k$ the Loeb measure on the subsets of $\UU^k$ obtained as the ultralimit of the uniform probability measures on the $\mathbf M_n^k$. Note that every internal subset of $\UU^k$ is $m_k$-measurable. According to {\L}o\'s's theorem, these include
    all the parameter-definable subsets of $\UU^k$. By \cite[Theorem 1.12]{keisler1977hyperfinite}, the Loeb measures satisfy Fubini. Hence, in our terminology, this means that $\mathbf U$ equipped with the probability measures $m_k$ is a measurable model. It is monadically stable, as it is a model of the common theory $T_\CC$.

     Since $\set{i | \pi(i) \geq n}=A_n \in F$,  for every formula $\varphi\in{\sf FO}_k$, we have that $m_k(\varphi(\UU))=\lim_{n\rightarrow\infty}\langle\varphi,\mathbf M_n\rangle$.

	According to Theorem~\ref{thm:tb} there exists a totally Borel model $\tb$ equipped with a Borel probability measure $\nu$, giving the same measure as $\UU$ to all definable sets.
	It follows that $\tb$ is a modeling ${\sf FO}$-limit of $(\mathbf M_n)_{n\in\mathbb N}$.
\end{proof}

Before continuing to  the proof of Theorem \ref{thm:removal} we take time out for a lemma.

\begin{lemma} \label{lem:rem1}
Let $(\mathbf B, \nu)$ be the modeling constructed in Theorem \ref{thm:tb}, and suppose that $\nu$ does not give positive measure to any singleton in $\NN$. For every formula $\varphi$ with $k$ free variables
the following are equivalent:
\begin{enumerate}
    \item there exist distinct $v_0,\dots,v_{k-1}\in \bI\times S\times\set{0}$ such that $\tb\models\varphi(\bar v)$;
    \item $\langle\varphi,\tb\rangle>0$.
\end{enumerate}
\end{lemma}
\begin{proof}
$(2) \Rightarrow (1)$ By our assumption that $\nu$ does not give positive measure to any singleton in $\NN$, it follows that $\nu$ is atomless. (The paragraph following the definition of $\nu$ in the proof of Theorem \ref{thm:tb} shows it does not give positive measure to any singleton outside $\NN$.) The result  is then immediate from the fact that the complement of $\bI \times S \times \set{0}$ has measure 0, and from the fact that $\nu$ is atomless. So we now continue to $(1) \Rightarrow (2)$. We proceed by induction on $k$, with the case $k=1$ immediate.
First note that if $v_0,\dots,v_{k-1}$ are distinct elements of $\bI\times S\times\set{0}$, then they are forking-independent.

\begin{claim*}
For any $\bbar = b_1, \dots, b_{k-1} \in \tb$, if there exists $v\in\bI\times S\times\set{0}$ such that $v \ind \bbar$ and
$\tb\models\varphi(v, \bbar)$, then
\[
\int I_\varphi(x, \bbar)\,{\rm d}\nu(x)>0.
\]
\end{claim*}
\begin{claimproof}
    By definability of types, there exists $\psi_{\bar b}(x)$ such that for all $x$ forking-independent from $\bbar$ we have
\[
\varphi(x,\bbar)\quad\iff\quad\psi_{\bar b}(x).
\]
As $\psi_{\bar b}\in\tp(v)$ and $\tp(v)$ belongs to the support of $\mu_1$, we have
$\mu_1(\clop(\psi_{\bar b}))>0$.
Thus, (as almost all $x$ are forking-independent from $\bbar$), we have 
\[
\int I_\varphi(x,\bbar)\,{\rm d}\nu(x)=
\int I_{\psi_{\bar b}}(x)\,{\rm d}\nu(x)=
\mu_1(\clop(\psi_{\bar b}))>0.
\]
\end{claimproof}

Let $A=\varphi(\tb,v_1,\dots,v_{k-1})$.
    By the assumption of $(1)$ and the Claim, we know that $\nu(A)>0$. 
    For each $a\in A$ there exists $\psi_a$ such that for every $\bar c$ forking-independent from $a$ we have $\varphi(a,\bar c)\leftrightarrow\psi_a(\bar c)$. As there are only countably many formulas, there exists some $\psi$ such that the set $A_\psi:=\set{a\in A |  \psi_a=\psi}$, which is Borel since $\psi_a$ is determined by $\tp_\varphi(a)$, has positive measure $\alpha$.
    Hence,
\[
\nu^{\otimes k}(\varphi(\tb))\geq\alpha\nu^{\otimes (k-1)}(\psi(\tb)).
\]
Moreover, $\tb\models\psi(v_1,\dots,v_{k-1})$. Thus, by the inductive hypothesis, $\nu^{\otimes (k-1)}(\psi(\tb))>0$. Hence, $\nu^{\otimes k}(\varphi(\tb))>0$.
\end{proof}

We now give a proof of our second main result.
\ThmRemoval*
\begin{proof}
	First note that, by modifying each relation of $\mathbf L$ on a null-set, we can produce a Borel structure $\mathbf L'$ such that $\langle\psi,\mathbf L\rangle=\langle\psi,\mathbf L'\rangle$ for every quantifier-free $\psi$.
	Let ${\rm Age}(\mathbf L')$ denote the \emph{age} of $\mathbf L'$, that is the class of the finite induced substructures of $\mathbf L'$.
	As $\mathbf L'$ is Borel, by \cite[Lemma~5]{QFTSL}, a random sampling of $\mathbf L'$ almost surely produces a sequence $(\mathbf M_n)_{n\in\mathbb N}$ of finite structure in ${\rm Age}(\mathbf L')$ left-converging to $\mathbf L'$.  Moreover, since $\mathbf L'$ only differs from $\mathbf L$ on a null set, almost surely every $\mathbf M_n$ is also in ${\rm Age}(\mathbf L)$. Thus, we can assume both that $(\mathbf M_n)_{n\in\mathbb N}$ is left-converging to $\mathbf L'$ and that 
    $\mathbf M_n\in {\rm Age}(\mathbf L)$.
    As $\mathbf L$ is monadically stable, the class ${\rm Age}(\mathbf L)$ (hence the class $\{\mathbf M_n\colon n<\omega\}$) is also monadically stable by \cite[Proposition 2.5]{braunfeld2022existential}.
    Up to passing to a subsequence, we may assume $(\mathbf M_n)_{n\in\mathbb N}$ is {\sf FO}-convergent, by compactness.
    
By Theorem \ref{thm:main}, let $\mathbf B$ be the modeling limit of $(\mathbf M_n)_{n\in\mathbb N}$ that we have constructed. By Lemma \ref{lem:rem1},
for every formula $\varphi$ with $k$ free variables with $\varphi\rightarrow\bigwedge_{i<j<k}(x_i\neq x_j)$  we have
\[
\varphi(\tb)\cap(\bI \times S \times \set{0})^k\neq\emptyset\quad\iff\quad \langle\varphi,\tb\rangle>0.
\]
Let $\tb'$ be the substructure of $\tb$ with domain $\bI \times S \times \set{0}$. Note that $\tb'$ is a Borel structure.
For every quantifier-free formula $\varphi$
with $k$ free variables with $\varphi\rightarrow\bigwedge_{i<j<k}(x_i\neq x_j)$,
we have
\[\langle\varphi,\tb'\rangle=\langle\varphi,\tb\rangle=\langle\varphi,\mathbf L\rangle=\langle\varphi,\mathbf L'\rangle.
\]
Moreover, 
$\varphi(\tb')\neq \emptyset$ implies $\varphi(\tb)\cap (\bI \times S \times \set{0})^k\neq\emptyset$, which implies $\langle\varphi,\tb\rangle>0$, thus 
$\langle\varphi,\tb'\rangle>0$.
\end{proof}

\section{Concluding remarks}
\label{sec:conc}
In this paper, we proved that every {\sf FO}-convergent sequence of structures from a monadically stable class has a modeling {\sf FO}-limit (Theorem~\ref{thm:main}). We conjecture that, for graphs, one can further require this modeling $\mathbf M$ to be a \emph{strong modeling}, meaning that it satisfies the \emph{strong finitary mass transport principle}: for all measurable subsets $A$ and $B$ of $\mathbf M$ such that each vertex of $A$ has at least $a$ neighbors in $B$ and each vertex in $B$ has at most $b$ neighbors in $A$, then
$a\,\nu(A)\leq b\,\nu(B)$.
For results on the existence of strong modeling limits for restricted classes of graphs, we refer the interested reader to \cite{gajarsky2016first, grzesik2021strong,limit1,modeling,MapLim}.

\begin{conjecture}
    Let $\mathscr C$ be a monadically stable class of graphs. Then, every {\sf FO}-convergent sequence of graphs in $\mathscr C$ has a strong modeling {\sf FO}-limit.
\end{conjecture}

We further believe that Theorem \ref{thm:main} (and Theorem \ref{thm:tb}) extends to the monadically NIP setting. 
\begin{conjecture}
    Let $\sigma$ be a countable signature, and 
let $\mathscr C$ be a  class of $\sigma$-structures. Then, the following properties are equivalent:
\begin{enumerate}
    \item Every {\sf FO}-convergent sequence of $\sigma$-structures in every unary expansion of $\mathscr C$ has a modeling {\sf FO}-limit.
    \item The class $\mathscr C$ is monadically NIP.
\end{enumerate}
\end{conjecture}
\bibliographystyle{amsplain}
\bibliography{ref}
\end{document}